\documentclass[graybox]{svmult}

\usepackage{mathptmx}       
\usepackage{helvet}         
\usepackage{courier}        
\usepackage{type1cm}        
\usepackage{makeidx}         
\usepackage{graphicx}        
\usepackage{multicol}        
\usepackage[bottom]{footmisc}

\usepackage{enumerate} 
\usepackage{amsmath, amssymb} 
\usepackage{url} 

\def\PP{{\cal P}} 

\makeindex             

\begin{document}

\title*{On the $L_2$ Markov Inequality with Laguerre Weight}
\author{Geno Nikolov and Alexei Shadrin}
\institute{Geno Nikolov \at Department of Mathematics and
Informatics, Sofia University ``St. Kliment Ohridski", 5 James
Bourchier Blvd., 1164 Sofia, Bulgaria \email{geno@fmi.uni-sofia.bg}
\and Alexei Shadrin \at Department of Applied Mathematics and
Theoretical Physics, Cambridge University, Wilberforce Road,
Cambridge CB3 0WA, United Kingdom \email{a.shadrin@damtp.cam.ac.uk}}
%
%
\maketitle

\abstract*{Let $w_{\alpha}(t)=t^{\alpha}\,e^{-t}$, $\alpha>-1$, be
the Laguerre weight function, and $\Vert\cdot\Vert_{w_\alpha}$
denote the associated $L_2$-norm, i.e.,
$$
\Vert f\Vert_{w_\alpha}:=\Big(\int_{0}^{\infty}w_{\alpha}(t)\vert
f(t)\vert^2\,dt\Big)^{1/2}.
$$
Denote by $\PP_n$ the set of algebraic polynomials of degree not
exceeding $n$. We study the best constant $c_n(\alpha)$ in the
Markov inequality in this norm,
$$
\Vert p^{\prime}\Vert_{w_\alpha}\leq c_n(\alpha)\,\Vert
p\Vert_{w_\alpha}\,,\quad p\in\PP_n\,,
$$
namely the constant
$$
c_{n}(\alpha)=\sup_{\mathop{}^{p\in\PP_n}_{p\ne 0}}\frac{\Vert
p^{\prime}\Vert_{w_\alpha}}{\Vert p\Vert_{w_\alpha}}\,,
$$
and we are also interested in its asymptotic value
$$
c(\alpha)=\lim_{n\rightarrow\infty}\frac{c_{n}(\alpha)}{n}\,.
$$
In this paper we obtain lower and upper bounds for both
$c_{n}(\alpha)$ and $c(\alpha)$. \\
Note that according to a result of P. D\"{o}rfler from 2002,
$c(\alpha)=[j_{(\alpha-1)/2,1}]^{-1}$, with $j_{\nu,1}$ being the
first positive zero of the Bessel function $J_{\nu}(z)$, hence our
bounds for $c(\alpha)$ imply bounds for $j_{(\alpha-1)/2,1}$ as
well. }

\abstract{
Let $w_{\alpha}(t)=t^{\alpha}\,e^{-t}$, $\alpha>-1$, be
the Laguerre weight function, and $\Vert\cdot\Vert_{w_\alpha}$
denote the associated $L_2$-norm, i.e.,
$$
\Vert f\Vert_{w_\alpha}:=\Big(\int_{0}^{\infty}w_{\alpha}(t)\vert
f(t)\vert^2\,dt\Big)^{1/2}.
$$
Denote by $\PP_n$ the set of algebraic polynomials of degree not
exceeding $n$. We study the best constant $c_n(\alpha)$ in the
Markov inequality in this norm,
$$
\Vert p^{\prime}\Vert_{w_\alpha}\leq c_n(\alpha)\,\Vert
p\Vert_{w_\alpha}\,,\quad p\in\PP_n\,,
$$
namely the constant
$$
c_{n}(\alpha)=\sup_{\mathop{}^{p\in\PP_n}_{p\ne 0}}\frac{\Vert
p^{\prime}\Vert_{w_\alpha}}{\Vert p\Vert_{w_\alpha}}\,,
$$
and we are also interested in its asymptotic value
$$
c(\alpha)=\lim_{n\rightarrow\infty}\frac{c_{n}(\alpha)}{n}\,.
$$
In this paper we obtain lower and upper bounds for both
$c_{n}(\alpha)$ and $c(\alpha)$. \\
Note that according to a result of P.
D\"{o}rfler from 2002, $c(\alpha)=[j_{(\alpha-1)/2,1}]^{-1}$, with
$j_{\nu,1}$ being the first positive zero of the Bessel function
$J_{\nu}(z)$, hence our bounds for $c(\alpha)$ imply bounds for
$j_{(\alpha-1)/2,1}$ as well.
}

\section{Introduction and Statement of the Results $L_p$}
\label{sec:1} The Markov inequality (or, to be more precise, the
inequality of the brothers Markov) has proven to be one of the most
important polynomial inequalities, with numerous applications in
approximation theory, numerical analysis, and many other branches of
mathematics. In its classical variant it reads as follows:\smallskip

\noindent {\bf The inequality of the brothers Markov.~} \emph{If
$p\in\PP_n$, then for $\,k=1,\ldots, n$,
$$
\Vert p^{(k)}\Vert \le T_n^{(k)}(1)\,\Vert p\Vert\,.
$$
The equality is attained if and only if $\,p=c\,T_n$, where
$\,T_n\,$ is the $n$-th Chebyshev polynomial of the first kind,
$\,T_n(x)=\cos n\arccos x,\ \ x\in [-1,1]$\,.}\smallskip

Here, $\,\PP_n$ is the set of algebraic polynomials of degree not
exceeding $\,n\,$ and $\,\Vert\cdot\Vert\,$ is the uniform norm in
$\,[-1,1]$, $\,\Vert f\Vert:=\sup\{ |f(x)|\, :\, x\in [-1,1]\}$.

Proved for $k=1$ in 1889 by Andrey Markov \cite{AM1889}, and for
$k\ge 1$, in 1892, by his kid brother, Vladimir  Markov
\cite{VM1892}, throughout the years Markov inequality has got many
alternative proofs and various generalizations. For the intriguing
story of Markov's inequality in the uniform norm, and twelve of its
proofs, we refer the reader to the survey paper \cite{AS2004}.
Another survey on the subject is \cite{BB1}. For some recent
developments, see \cite{BN1996, GN1998, GN2001, GN2005, GN2005b,
GN2005A, NS2012, NS2014}.

Generally, Markov-type inequalities provide upper bounds for
a certain norm of a derivative of an algebraic polynomial
$p\in\PP_n$ in terms of some (usually the same) norm of $p$.
Our subject here is
Markov-type inequalities in $L_2$-norms for the first derivative of
an algebraic polynomial. For a weight function $w$ on the finite
or infinite interval $\,(a,b)\,$ with all moments finite, let
$\,\Vert\cdot\Vert_{w}\,$ be the associated $L_2$-norm,
$$
\Vert f\Vert_{w}:=\Big(\int_{a}^{b}w(t)\vert
f(t)\vert^2\,dt\Big)^{1/2},
$$
and let $c_n(w)$ be the best (i.e., the smallest) constant in the $L_2$
Markov inequality
$$
\Vert p^{\prime}\Vert_{w}\leq c_{n}(w)\,\Vert p\Vert_{w},\qquad p\in
\PP_n\,.
$$
This constant possesses a simple characterization: it is the largest
singular value of a certain matrix, see, e.g., \cite{PD1987} or
\cite{MMR1994}, however the exact values of the best Markov
constants are generally unknown even in the cases of the classical
weight functions of Laguerre and Jacobi, and, in particular, of
Gegenbauer.\smallskip

{\bf The Hermite weight} \boldmath{ $w_H(t) = e^{-t^2},\, t\in
\mathbb{R}.$} This is the only case where both the sharp Markov
constant and the extremal polynomial are known. Namely, in this case
the sharp Markov constant is $c_n(w_H)=\sqrt{2n}$\,, and the unique
(up to a constant factor) extremal polynomial is the $n$-th Hermite
polynomial $H_n(t)= (-1)^n\,
e^{t^2}\big(\frac{d}{dt}\big)^n\,e^{-t^2}$. The extremality of $H_n$
persists in the $L_2$ Markov inequalities for higher order
derivatives,
$$
\Vert p^{(k)}\Vert_{w_H}\le c_n^{(k)}(w_H)\,\Vert p\Vert_{w_H}\,,\quad
k=1,\ldots,n\,,
$$
with the sharp Markov constants given by $
c_n^{(k)}(w_H)=\Big(2^{k}\,\frac{n!}{(n-k)!}\Big)^{1/2}$\,. The
reason for this case to be trivial comes from the fact that the
derivatives of Hermite's polynomials are Hermite's polynomials of
lower degrees \cite[Chapt. 5]{GS}, and as a result, the sharp Markov
constant is simply the largest entry in a diagonal matrix.\medskip

\textbf{The Gegenbauer weight} \boldmath{
$\,w_{\lambda}(t)=(1-t^2)^{\lambda-1/2}\,$},
\boldmath{$\lambda>-1/2,\,t\in [-1,1]$.} Neither the sharp Markov
constant nor the extremal polynomial are known explicitly in that
case. For $\lambda=1/2$ (a constant weight function) E. Schmidt
\cite{ES} found tight estimates for the Markov constant, which in a
slightly weaker form look like
$$
   \frac{1}{\pi}(n+3/2)^2 < c_n(w_{1/2}) < \frac{1}{\pi}(n+2)^2\,,\qquad n>5.
$$
Recently, A. Kro\'{o} \cite{AK2008} turned back
to this case, identifying $c_{n}(w_{1/2})$ as the largest positive
root of a polynomial of degree $n$. This polynomial was found
explicitly (to some extent) by Kro\'{o}.

Nikolov \cite{GN2003} studied two further
special cases $\lambda=0$ and $\lambda=1$; in particular, he
obtained the following two-sided estimates for the corresponding
best Markov constants:
\begin{eqnarray*}
   & 0.472135\, n^2 \le c_{n}(w_0) \le 0.478849\, (n+2)^2\,, & \\
   & 0.248549\, n^2 \le c_{n}(w_1) \le 0.256861\, (n+\frac{5}{2})^2\,. &
\end{eqnarray*}

In \cite{SNA2015} we obtained an upper bound for $c_n(w_{\lambda})$,
which is valid for all $\lambda>-1/2$:
$$
c_{n}(w_{\lambda})< \frac{(n+1)(n+2\lambda+1)}{2\sqrt{2\lambda+1}}
\,,
$$
however it seems that the correct order with respect to $\lambda$
should be $O(1/\lambda)$. Also, it has been shown in \cite{SNA2015}
that the extremal polynomial in the $L_2$ Markov inequality
associated with $w_{\lambda}\,$, is even or odd when $n$ is even or
odd, accordingly (for $\lambda\ge 0$ this result was established, by
a different argument, in \cite{GN2003}).

\textbf{The Laguerre weight} $\,w_{\alpha}(t)=t^{\alpha}e^{-t},\ \
t\in (0,\infty)\,,\ \alpha>-1\,$. In the present paper we study the
best constant in the Markov inequality for the first derivative of
an algebraic polynomial in the $L_2$-norm, induced by the Laguerre
weight function. We denote this norm by
$\Vert\cdot\Vert_{w_{\alpha}}$,
\begin{equation}\label{e1}
\Vert
f\Vert_{w_{\alpha}}
:=\Big(\int_{0}^{\infty}t^{\alpha}e^{-t}|f(t)|^2\,dt\Big)^{1/2}\,.
\end{equation}

Further, we denote by $c_n(\alpha)$ the best constant in the Markov
inequality in this norm,
\begin{equation}\label{e2}
c_{n}(\alpha)=\sup_{\mathop{}^{p\in\PP_n}_{p\ne 0}}\frac{\Vert
p^{\prime}\Vert_{w_{\alpha}}}{\Vert p\Vert_{w_{\alpha}}}\,.
\end{equation}

Before formulating our results, let us give a brief account on the
known results on the Markov inequality in the $L_2$ norm induced by
the Laguerre weight function. P. Tur\'{a}n \cite{PT1960} found the
sharp Markov constant in the case $\alpha=0$, namely,
\begin{equation}\label{e3}
c_n(0)=\Big(2\sin\frac{\pi}{4n+2}\Big)^{-1}\,.
\end{equation}

In 1991, D\"{o}rfler \cite{PD1991} proved the inequalities
\begin{equation}\label{e4}
\frac{n^2}{(\alpha+1)(\alpha+3)}\leq
\big[c_n(\alpha)\big]^{2}\le\frac{n(n+1)}{2(\alpha+1)}\,,
\end{equation}
(the first one in a somewhat stronger form), and in 2002 he found
\cite{PD2002} the sharp asymptotic of $c_n(\alpha)$, namely,
\begin{equation}\label{e5}
c(\alpha):=\lim_{n\rightarrow\infty}\frac{c_n(\alpha)}{n}=\frac{1}{j_{(\alpha-1)/2,1}}\,,
\end{equation}
where $j_{\nu,1}$ is the first positive zero of the Bessel function
$J_{\nu}(z)$\,.

In a series of recent papers \cite{BD2010, BD2010a, BD2011} A.
B\"{o}ttcher and P. D\"{o}rfler studied the asymptotic values of the
best constants in $L_2$ Markov-type inequalities of a rather general
form, namely 1) they include estimates for higher order derivatives
and 2) different $L_2$-norms of Laguerre or Jacobi type are applied
to the polynomial and its derivatives (i.e. at the two sides of
their Markov inequalities).

Precisely, they proved that those asymptotic values are equal to the
norms of certain Volterra operators. It seems, however, that finding
the norms of these related Volterra operators explicitly is equally
difficult task. They provide also some upper and lower bounds for
the asymptotic values, but they do not match (they are similar to
those given in \eqref{e4}).

Our main goal is upper and lower bounds for the Markov constant $c_n(\alpha)$
which are valid for all $n$ and $\alpha$.
\smallskip

In this paper we prove the following.

%
%
\begin{theorem} \label{t2}
For all $\alpha>-1\,$ and $\,n\in \mathbb{N}\,$, $\,n\geq 3\,$, the
best constant $\,c_n(\alpha)\,$ in the Markov inequality
$$
\Vert p^{\prime}\Vert_{w_{\alpha}} \leq c_n(\alpha)\,\Vert
p\Vert_{w_{\alpha}}\,,\qquad p\in\PP_n
$$
admits the estimates
$$
   \frac{ 2 \big(n+\frac{2\alpha}{3}\big)
         \big(n-\frac{\alpha+1}{6}\big)}
     {(\alpha+1)(\alpha+5)}
< \big[c_n(\alpha)\big]^2 <
  \frac{\big(n+1\big) \big(n+\frac{2(\alpha+1)}{5}\big)}
  {(\alpha+1)\big((\alpha+3)(\alpha+5)\big)^{\frac{1}{3}}}\,,
$$
where for the left-hand inequality it is additionally assumed that
$n>(\alpha+1)/6$\,.
\end{theorem}

For $n=1,\,2$, the exact values are readily computable:
$$
\big[c_1(\alpha)\big]^2=\frac{1}{1+\alpha}\,,\qquad
\big[c_2(\alpha)\big]^2=\frac{3(\alpha+2)+\sqrt{(\alpha+2)(\alpha+10)}}
{2(\alpha+1)(\alpha+2)}\,.
$$
Compared to D\"orfler's result \eqref{e4}, we improve the lower
bound for $c_n(\alpha)$ by the factor of $\sqrt{2}$, and obtain for
the upper bound the order $\,O(n/\alpha^{5/6}\,)$ instead of
$\,O(n/\alpha^{1/2}\,)$.

As an immediate consequence of Theorem \ref{t2} we obtain the
following
\begin{corollary}\label{c1}
The asymptotic Markov constant
$c(\alpha)=\lim_{n\rightarrow\infty}\{n^{-1}\,c_n(\alpha)\}$
satisfies the inequalities
\begin{equation}\label{e8}
\underline{c}(\alpha):=\frac{\sqrt{2}}{\sqrt{(\alpha+1)(\alpha+5)}}\leq
c(\alpha)\leq
\frac{1}{\sqrt{\alpha+1}\sqrt[6]{(\alpha+3)(\alpha+5)}}=:\overline{c}(\alpha)\,.
\end{equation}
\end{corollary}

Let us comment now on the bounds for $c(\alpha)$ given by Corollary
\ref{c1}. First of all,
$$
\lim_{\mathop{}^{\alpha\rightarrow -1}_{\alpha>-1}}\,
\frac{\overline{c}(\alpha)}{\underline{c}(\alpha)}=1\,,
$$
which indicates that for small $\alpha$ our bounds are pretty tight.
In particular, in the case $\alpha=0$, when we have $\,c(0)=2/\pi\,$
(see \eqref{e3}), the relative errors satisfy
$$
\frac{c(0)}{\underline{c}(0)}=\frac{\sqrt{10}}{\pi}<1.006585\,,\qquad
\frac{\overline{c}(0)}{c(0)}=\frac{\pi}{2\sqrt[6]{15}}<1.000242\,.
$$

Second, Corollary \ref{c1} gives rise to the question: what is the
right order of $\alpha$ in $c(\alpha)$ as $\alpha\rightarrow\infty$
?  The answer follows below:

\begin{theorem} \label{t3}
For the asymptotic Markov constant $c(\alpha)$ we have
$c(\alpha)=O(\alpha^{-1})$ as $\alpha\rightarrow\infty$\,. More
precisely, $\,c(\alpha)\,$ satisfies the inequalities
\begin{equation}\label{e9}
\frac{\sqrt{2}}{\sqrt{(\alpha+1)(\alpha+5)}}<c(\alpha)<\frac{2}{\alpha+2\pi-2}\,,
\quad\alpha>1\,.
\end{equation}
\end{theorem}
\begin{proof}
The lower bound for $\,c(\alpha)\,$ is simply
$\underline{c}(\alpha)$ (in fact, the left-hand inequality in
\eqref{e9} holds for all $\alpha>-1$). For the right-hand inequality
in \eqref{e9}, we recall that, by D\"{o}rfler's result \eqref{e5},
$\,c(\alpha)=\big[j_{(\alpha-1)/2,1}\big]^{-1}$, with $j_{\nu,1}$
being the first positive zero of the Bessel function $J_{\nu}(z)$\,.
On using some lower bounds for the zeros of Bessel functions,
obtained by Ifantis and Siafarikas \cite{IS1985} (see \cite[eqn.
(1.6)]{AE2001}), we get
$$
\frac{1}{j_{(\alpha-1)/2,1}}<\frac{2}{\alpha+2\pi-2}\,,\quad
\alpha>1\,.
$$
The inequalities in \eqref{e9} imply that
$\,c(\alpha)=O(\alpha^{-1})\,$ as $\,\alpha\rightarrow\infty$\,.
\qed
\end{proof}

Notice that the lower bound $\underline{c}(\alpha)$ has the right
order with respect to $\alpha$ as $\alpha\rightarrow\infty$.
Moreover, from \eqref{e9} it follows that, roughly, this lower bound
can only be improved by a factor of at most $\sqrt{2}$.

\begin{figure}[h]
\sidecaption
\includegraphics[scale=0.9]{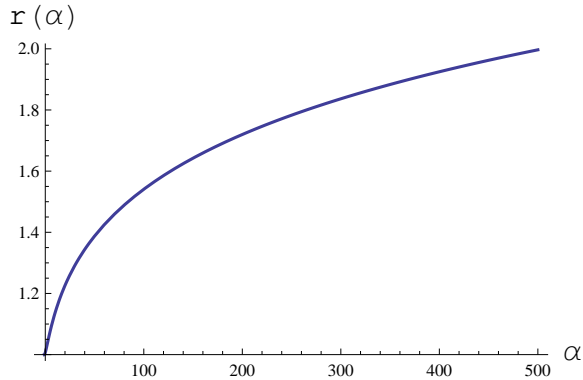}
\caption{The graph of the ratio
$r(\alpha)=\displaystyle{\frac{\overline{c}(\alpha)}{\underline{c}(\alpha)}}$.}
\label{f1}
\end{figure}

The upper bound $\overline{c}(\alpha)$ does not exhibit the right
asymptotic of $c(\alpha)$ as $\alpha\rightarrow\infty$.
Nevertheless, $\overline{c}(\alpha)$ is less than the upper bound in
\eqref{e9} for $\alpha\in [2.045,47.762]$. Moreover, the ratio
$r(\alpha)=\overline{c}(\alpha)/\underline{c}(\alpha)$ tends to
infinity as $\alpha\rightarrow\infty$ rather slowly; for instance,
$\,r(\alpha)\,$ is less than two for $\,-1<\alpha<500\,$ (see Fig.
1).\smallskip

Finally, in view of \eqref{e5}, Corollary \ref{c1} provides bounds
for $j_{\nu,1}$, the first positive zero of the Bessel function
$J_{\nu}$, which, for some particular values of $\nu$, are better
than some of the bounds known in the literature (e.g., the lower
bound below is better than the one given in \cite[eqn.
(1.6)]{AE2001} for $\nu\in [0.53, 23.38]$).
\begin{corollary}\label{c2} The first positive zero $\,j_{\nu,1\,}$ of
the Bessel function $\,J_{\nu}$, $\nu>-1$, satisfies the
inequalities
$$
2^{\frac{5}{6}}\,\sqrt{\nu+1}\,\sqrt[6]{(\nu+2)(\nu+3)}<j_{\nu,1}
<\sqrt{2(\nu+1)(\nu+3)}\,.
$$
\end{corollary}

The rest of the paper is organized as follows. In Sect. \ref{s2} we
present some preliminary facts, which are needed for the proof of
Theorem \ref{t2}. In Sect. \ref{s2.1} we quote a known relation
between the best Markov constant $c_{n}(\alpha)$ and the smallest
(positive) zero of a polynomial $Q_n(x)=Q_n(x,\alpha)$ of degree
$n$, defined by a three-term recurrent relation. By this definition,
$Q_n$ is identified as an orthogonal polynomial with respect to a
measure supported on $\mathbb{R}_{+}$. In Sect. \ref{s2.2} we give
lower and upper bounds for the largest zero of a polynomial, which
has only positive zeros, in terms of a few of its highest degree
coefficients. In Sect. 3 we prove formulae for the four lowest
degree coefficients of the polynomial $Q_n$. The proof of our main
result, Theorem~\ref{t2}, is given in Sect. \ref{s4}. As the proof
involves some lengthy tough straightforward calculations, for
performing part of them we have used the assistance of a computer
algebra system. Section \ref{s5} contains some final remarks.
\section{Preliminaries}\label{s2}
In this section we quote some known facts, and prove some results
which will be needed for the proof of Theorem \ref{t2}.

\subsection{A Relation Between \boldmath{$c_n(\alpha)$} and an Orthogonal
Polynomial}\label{s2.1}

As was already said in the introduction, the best constant in a
$L_2$ Markov inequality for polynomials of degree not exceeding $n$
is equal to the largest singular value of a certain $n\times n$
matrix, say $\mathbf{A}_n$. The latter is equal to a square root of
the largest eigenvalue of $\mathbf{A}_n^{\top}\mathbf{A}_n$ (or
$\Vert \mathbf{A}_n\Vert_2$, the second matrix norm of
$\mathbf{A}_n$). However, finding explicitly $\Vert
\mathbf{A}_n\Vert_2$ (and for all $n\in \mathbb{N}$) is a fairly
difficult task, and this explains the lack of many results on the
sharp constants in the $L_2$ Markov inequalities. To avoid this
difficulty, some authors simply try to estimate $\Vert
\mathbf{A}_n\Vert_2$, or use other matrix norms, e.g., $\,\Vert
\mathbf{A}_n\Vert_{\infty}$, the Frobenius norm, etc.\smallskip

Our approach to the proof of Theorem \ref{t2} makes use of the
following theorem:
\begin{theorem}[\textbf{\cite[p. 85]{PD2002}}] \label{t4}
The quantity $\,1/[c_n(\alpha)]^2\,$ is equal to the smallest zero
of the polynomial $\,Q_n(x)=Q_n(x,\alpha)$, which is defined
recursively by
\begin{eqnarray*}
&&Q_{n+1}(x)=(x-d_n)Q_n(x)-\lambda_n^2 Q_{n-1}(x),\quad n\ge 0\,;\\
&&Q_{-1}(x):=0,\ \ Q_0(x):=1\,; \\
&&d_0:=1+\alpha,\ \ d_n:=2+\frac{\alpha}{n+1}\,,\quad n\ge 1\,;\\
&&\lambda_0>0\ \ \text{{\rm arbitrary}},\
\lambda_n^2:=1+\frac{\alpha}{n}\,\quad n\ge 1\,.
\end{eqnarray*}
\end{theorem}

By Favard's theorem, for any $\alpha>-1$,
$\{Q_n(x,\alpha)\}_{n=0}^{\infty}$ form a system of monic orthogonal
polynomials, and, in addition, we know that the support of their
ortogonality measure is in $\mathbb{R}_{+}$. Theorem \ref{t4}
transforms the problem of finding or estimating $c_n(\alpha)$ to a
problem for finding or estimating the extreme zeros of orthogonal
polynomials, or, equivalently, the extreme eigenvalues of certain
tri-diagonal (Jacobi) matrices. For the latter problem one can apply
numerous powerful methods such as the Gershgorin circles, the ovals
of Cassini, etc. For more details on this kind of methods we refer
the reader to the excellent paper of van Doorn \cite{EvD1987}.

However, we choose  here a different approach for estimating the
smallest positive zero of $\,Q_{n}(x,\alpha)$, which seems to be
efficient, too.

\subsection{Bounds for the Largest Zero of a Polynomial Having Only Positive Roots}
\label{s2.2}

In view of Theorem \ref{t4}, we need to estimate the smallest
(positive) zero of the polynomial $Q_{n}(x,\alpha)$. On using the
three-term recurrence relation for $\{Q_{m}\}_{m=0}^{\infty}$, we
can evaluate (at least theoretically) as many coefficients of
$Q_n(x)$ as we wish (and thus coefficients of the reciprocal
polynomial $x^{n}\,Q_n(x^{-1})$, too). Our proof of Theorem \ref{t2}
makes use of the following statement.

\begin{proposition}\label{p1}
Let
$P(x)=x^n-b_{1}\,x^{n-1}+b_{2}\,x^{n-2}-\cdots+(-1)^{n-1}b_{n-1}\,x+(-1)^{n}b_n$
be a polynomial with positive roots $\,x_1\le x_2\le\cdots\le
x_n$\,. Then the largest zero $x_n$ of $P$ satisfies the
inequalities:
\begin{enumerate}[~(i)~~~~]
\item $\displaystyle{\frac{b_1}{n}\leq x_n< b_1}\,$;
\item $\displaystyle{b_1-2\,\frac{b_2}{b_1}\leq x_n< (b_1^2-2b_2)^{\frac{1}{2}}}\,$;
\item $\displaystyle{\frac{b_1^{3}-3b_1\,b_2+3b_3}{b_1^2-2b_2}\leq
x_n < (b_1^{3}-3b_1\,b_2+b_3)^{\frac{1}{3}}}$\,.
\end{enumerate}
\end{proposition}

\begin{proof}
Part (i) follows trivially from
$$
\frac{b_1}{n}=\frac{x_1+x_2+\cdots+x_n}{n}\leq x_n<
x_1+x_2+\cdots+x_n =b_1\,.
$$
For the proof of parts (ii) and (iii) we make use of Newton's
identities to obtain
$$
x_1^2+x_2^2+\cdots+x_n^2=b_1^2-2b_2,\qquad
x_1^3+x_2^3+\cdots+x_n^3=b_1^3-3b_1\,b_2+3b_3\,.
$$
Now (ii) follows from
$$
\frac{b_1^2-2b_2}{b_1}=\frac{x_1^2+x_2^2+\cdots+x_n^2}{x_1+x_2+\cdots+x_n}\leq
x_n<
(x_1^2+x_2^2+\cdots+x_n^2)^{\frac{1}{2}}=(b_1^2-2b_2)^{\frac{1}{2}}
$$
and (iii) follows from
$$
\frac{b_1^3-3b_1\,b_2+3b_3}{b_1^2-2b_2}=\frac{x_1^3+\cdots+x_n^3}{x_1^2+\cdots+x_n^2}\le
x_n<
(x_1^3+\cdots+x_n^3)^{\frac{1}{3}}=(b_1^3-3b_1\,b_2+3b_3)^{\frac{1}{3}}\,.
$$

It is clear from the proof that the lower bounds for $x_n$ are
attained only when $x_1=x_2=\cdots=x_n$. \qed
\end{proof}

\section{The Lowest Degree Coefficients of the Polynomial \boldmath{$Q_{n,\alpha}$}}
\label{s3}

Let us denote by  $a_{k,n}=a_{k,n}(\alpha)$, $k=0,\ldots,n$, the
coefficients of the monic polynomial $Q_n(x)=Q_n(x,\alpha)$,
introduced in Theorem \ref{t4}, i.e.,
$$
Q_n(x)=Q_n(x,\alpha)=x^n+a_{n-1,n}\,x^{n-2}+\cdots+a_{3,n}\,x^3+a_{2,n}\,x^2
+a_{1,n}\,x+a_{0,n}\,.
$$

For the sake of convenience, we set $a_{m,m}=1$, $m\geq 0$, and
$$
a_{k,m}=0\,,\ \text{ if }\ k<0\ \text{ or }\ k>m\,.
$$
From the recursive definition of $Q_n$ we have $\,Q_0(x)=1\,$,
$\,Q_1(x)=x-\alpha-1\,$, thus
$$
a_{0,1}=-\alpha-1\,,
$$
and for $n\in \mathbb{N}$ we obtain a recurrence relations for the
coefficients of $Q_{n-1}$, $Q_n$ and $Q_{n+1}$\,:
\begin{equation}\label{e10}
a_{k,n+1}=a_{k-1,n}-\Big(2+\frac{\alpha}{n+1}\Big)\,a_{k,n}
-\Big(1+\frac{\alpha}{n}\Big)a_{k,n-1}\,,\quad k=0,\ldots,n\,.
\end{equation}

Now recurrence relation \eqref{e10} will be used of proving
consecutively formulae for the coefficients $\,a_{k,n}\,$, $\,0\leq
k\leq 3$.

\begin{proposition}\label{p2}
For all $\,n\in \mathbb{N}_0$, the coefficient $\,a_{0,n}\,$ of the
polynomial $\,Q_n\,$ is given by
$$
a_{0,n}=(-1)^{n}\,\prod_{k=1}^{n}\Big(1+\frac{\alpha}{k}\Big)\,.
$$
\end{proposition}

\begin{proof}
We apply induction with respect to $n$. Since $a_{0,0}=1$ and
$a_{0,1}=-(1+\alpha)$, Proposition \ref{p2} is true for $n=0$ and
$n=1$. For $k=0$ the recurrence relation \eqref{e10} becomes
$$
a_{0,n+1}=-\Big(2+\frac{\alpha}{n+1}\Big)\,a_{0,n}
-\Big(1+\frac{\alpha}{n}\Big)a_{0,n-1}\,,\quad n\in \mathbb{N}\,.
$$
Assuming Proposition \ref{p2} is true for $m\leq n$, for $m=n+1$ we
obtain
\[
\begin{split}
a_{0,n+1}&=-\Big(2+\frac{\alpha}{n+1}\Big)\,
(-1)^{n}\,\prod_{k=1}^{n}\Big(1+\frac{\alpha}{k}\Big)
-\Big(1+\frac{\alpha}{n}\Big)\,(-1)^{n-1}\prod_{k=1}^{n-1}\Big(1+\frac{\alpha}{k}\Big)\\
&=(-1)^{n+1}\prod_{k=1}^{n+1}\Big(1+\frac{\alpha}{k}\Big)\,,
\end{split}
\]
hence the induction step is done, and Proposition \ref{p2} is
proved. \qed
\end{proof}

Before proceeding with the proof of the formulae for $a_{k,n}$,
$1\leq k\leq 3$, let us point out to the relation
\begin{equation}\label{e11}
a_{0,m+1}=-\Big(1+\frac{\alpha}{m+1}\Big)\,a_{0,m}\,,\quad m\in
\mathbb{N}_0\,,
\end{equation}
which follows from Proposition \ref{p2}, and will be used in the
proof of the next propositions.

\begin{proposition} \label{p3}
For all $\,n\in \mathbb{N}_0$, the coefficient $\,a_{1,n}\,$ of the
polynomial $\,Q_n\,$ is given by
$$
a_{1,n}=-\frac{n(n+1)}{2(\alpha+1)}\;a_{0,n}\,.
$$
\end{proposition}

\begin{proof}
Again, we apply induction on $n$. Proposition \ref{p3} is true for
$n=0$ and $n=1$. Indeed, by our convention, $\,a_{1,0}=0$, and
$a_{1,1}=1$ also obeys the desired representation, as
$a_{0,1}=-(1+\alpha)$\,. Assume that Proposition \ref{p3} is true
for $\,m\leq n$, $\,m\in \mathbb{N}$. From the recurrence relation
\eqref{e10} (with $\,k=1$), the induction hypothesis and \eqref{e11}
we obtain
\[
\begin{split}
a_{1,n+1}&=a_{0,n}-\Big(2+\frac{\alpha}{n+1}\Big)\;a_{1,n}
-\Big(1+\frac{\alpha}{n}\Big)\;a_{1,n-1}\\
&=a_{0,n}+\Big(2+\frac{\alpha}{n+1}\Big)\,\frac{n(n+1)}{2(\alpha+1)}\;a_{0,n}
+\Big(1+\frac{\alpha}{n}\Big)\,\frac{(n-1)n}{2(\alpha+1)}\;a_{0,n-1}\\
&=a_{0,n}\,\Big[1+\Big(2+\frac{\alpha}{n+1}\Big)\,\frac{n(n+1)}{2(\alpha+1)}-
\frac{(n-1)n}{2(\alpha+1)}\Big]\\
&=\frac{a_{0,n}}{2(\alpha+1)}\,\big[n^2+(\alpha+3)n+2(\alpha+1)\big]
=a_{0,n}\,\frac{(n+2)(n+\alpha+1)}{2(\alpha+1)}\\
&=\frac{(n+1)(n+2)}{2(\alpha+1)}\,\Big(1+\frac{\alpha}{n+1}\Big)\;a_{0,n}
=-\frac{(n+1)(n+2)}{2(\alpha+1)}\;a_{0,n+1}\,.
\end{split}
\]
Hence,  the induction step is done, and the proof of Proposition
\ref{p3} is complete. \qed
\end{proof}

\begin{proposition} \label{p4}
For all $\,n\in \mathbb{N}_0$, the coefficient $\,a_{2,n}\,$ of the
polynomial $\,Q_n\,$ is given by
$$
a_{2,n}=\frac{(n-1)n(n+1)}{24(\alpha+1)(\alpha+2)(\alpha+3)}
\;\big[3(\alpha+2)n+2(\alpha+6)\big]\;a_{0,n}\,.
$$
\end{proposition}

\begin{proof}
The claim is true for $n=0,\,1$ (according to our convention), and
also for $n=2$, as in this case, taking into account that
$a_{0,2}=1/\big((1+\alpha)(1+\alpha/2)\big)$,  the above formula
produces $a_{2,2}=1$. Assume now that the proposition is true for
$m\leq n$, where $n\in \mathbb{N}$,  $\,n\ge 2\,$. We shall prove
that it is true for $m=n+1$, too, thus proving Proposition \ref{p4}
by induction. On using the recurrence relation \eqref{e10} (with
$k=2$), the inductional hypothesis, Proposition \ref{p3} and
\eqref{e11} we obtain
\[
\begin{split}
a_{2,n+1}&=a_{1,n}-\Big(2+\frac{\alpha}{n+1}\Big)\;a_{2,n}
-\Big(1+\frac{\alpha}{n}\Big)\;a_{2,n-1}\\
&=-\frac{n(n+1)}{2(\alpha+1)}\;a_{0,n}-\Big(2+\frac{\alpha}{n+1}\Big)\,
\frac{(n\!-\!1)n(n\!+\!1)\big[3(\alpha\!+\!2)n
\!+\!2(\alpha\!+\!6)\big]}{24(\alpha+1)(\alpha+2)(\alpha+3)}\;a_{0,n}\\
&\quad +\frac{(n\!-\!2)(n\!-\!1)n\big[3(\alpha\!+\!2)(n\!-\!1)
\!+\!2(\alpha\!+\!6)\big]}{24(\alpha+1)(\alpha+2)(\alpha+3)}\;a_{0,n}\\
&=\frac{n(n+1)}{n+\alpha+1}\Bigg[\frac{n+1}{2(\alpha+1)}+
\Big(2+\frac{\alpha}{n+1}\Big)\,\frac{(n^2-1)\big[3(\alpha+2)n+2(\alpha+6)\big]}
{24(\alpha+1)(\alpha+2)(\alpha+3)}\\
&\qquad\qquad\qquad -
\frac{(n-2)(n-1)\big[3(\alpha+2)(n-1)+2(\alpha+6)\big]}
{24(\alpha+1)(\alpha+2)(\alpha+3)}\Bigg]\;a_{0,n+1}\,.
\end{split}
\]
After some calculations the expression in the big brackets
simplifies to
$$
\frac{(n+2)(n+\alpha+1)\big[(3(\alpha+2)(n+1)+2(\alpha+6)\big]}
{24(\alpha+1)(\alpha+2)(\alpha+3)}\,.
$$
and substitution of this expression yields the desired formula for
$a_{2,n+1}$. The induction proof of Proposition \ref{p4} is
complete.  \qed
\end{proof}

\begin{proposition} \label{p5}
For all $\,n\in \mathbb{N}_0$, the coefficient $\,a_{3,n}\,$ of the
polynomial $\,Q_n\,$ is given by
$$
a_{3,n}\!=\!\frac{-(n\!-\!2)(n\!-\!1)n(n\!+\!1)
\big[5(\alpha\!+\!2)(\alpha\!+\!4)n(n\!+\!1)\!+\!
8(7\alpha\!+\!20)n\!+\!12(\alpha\!+\!20)\big]}
{240(\alpha+1)(\alpha+2)(\alpha+3)(\alpha+4)(\alpha+5)}\,a_{0,n}.
$$
\end{proposition}

\begin{proof}
Again, induction is applied with respect to $n$. The formula for
$a_{3,n}$ is easily verified to be true for $\,0\leq n\leq 3$. Then,
assuming that this formula is true for $\,m\leq n\,$,  where $\,n\in
\mathbb{N}$, $n\geq 3$, we prove that it is true also for $m=n+1$,
too. The induction step is performed along the same lines as the one
in the proof of Proposition \ref{p4}. First, we make use of the
recurrence relation \eqref{e10} with $k=3$ to express $a_{3,n+1}$ as
a linear combination of $a_{2,n}$, $a_{3,n}$ and $a_{3,n-1}$. Next,
we apply the inductional hypothesis and \eqref{e11} to represent
$a_{3,n+1}$ in the form
$$
a_{3,n+1}=\frac{-(n-1)n(n+1)}{240(\alpha+1)(\alpha+2)(\alpha+3)
(\alpha+4)(\alpha+5)}\,\frac{r(n)}{n+\alpha+1}\;a_{0,n+1}\,,
$$
where $r(n)=r(n,\alpha)$ is a polynomial of $4$-th degree. With some
lengthy tough straightforward calculation (we used a computer
algebra program for verification) we obtain that
$$
r(n)\!=\!(n\!+\!2)(n\!+\!\alpha\!+\!1)
\big[5(\alpha\!+\!2)(\alpha\!+\!4)(n\!+\!1)(n\!+\!2)\!+\!
8(7\alpha\!+\!20)(n\!+\!1)\!+\!12(\alpha\!+\!20)\big]
$$
and this expression substituted in the above formula implies the
desired representation of $a_{3,n+1}$. To keep the paper condensed,
we omit the details. \qed
\end{proof}

\section{Proof of Theorem \ref{t2}}
\label{s4}

For the proof of Theorem \ref{t2} we prefer to work with the
(constant multiplier of) reciprocal polynomial of  $Q_n$
$$
P_n(x)=P_n(x,\alpha)=
(-1)^{n}\,\big(a_{0,n}\big)^{-1}\,x^{n}\,Q_n\big(x^{-1}\big)\,.
$$
Clearly, $P_n$ is a monic polynomial of degree $n$,
$$
P_n(x)=x^{n}-b_1\,x^{n-1}+b_2\,x^{n-2}-b_3\,x^{n-3}+\cdots
$$
and, in view of Propositions \ref{p2}--\ref{p5}, its coefficients
$b_1$, $b_2$ and $b_3$ are
\begin{eqnarray*}
&&b_1=\frac{n(n+1)}{2(\alpha+1)}\,,\quad
b_2=\frac{(n-1)n(n+1)}{24(\alpha+1)(\alpha+2)(\alpha+3)}
\;\big[3(\alpha+2)n+2(\alpha+6)\big]\,,\\
&& b_3=\frac{(n\!-\!2)(n\!-\!1)n(n\!+\!1)
\big[5(\alpha\!+\!2)(\alpha\!+\!4)n(n\!+\!1)\!+\!
8(7\alpha\!+\!20)n\!+\!12(\alpha\!+\!20)\big]}
{240(\alpha+1)(\alpha+2)(\alpha+3)(\alpha+4)(\alpha+5)}\,.
\end{eqnarray*}

As was said in Sect. \ref{s2.1}, $Q_n(x,\alpha)$ is identified an
orthogonal polynomial with positive and distinct zeros. Therefore,
the same can be said for the zeros of $P_n$ (as reciprocal of
$Q_n$). If $x_n$ is the largest zero of $P_n$, then, according to
Theorem \ref{t4}, we have $\,\big[c_n(\alpha)\big]^2=x_n$.\smallskip

Now Proposition \ref{p1} (iii) applied to $P=P_n$ yields immediately
the following
\begin{proposition} \label{p6}
For all $\,n\in \mathbb{N}\,$, $n\geq 3\,$, the best Markov constant
$c_n(\alpha)$ satisfies
$$
\frac{b_1^3-3b_1\,b_2+3b_3}{b_1^2-2b_2}<\big[c_n(\alpha)\big]^2 <
(b_1^3-3b_1\,b_2+3b_3)^{\frac{1}{3}}
$$
with $b_1$, $b_2$ and $b_3$ as given above.
\end{proposition}

The estimates for $c_n(\alpha)\big]$ in Theorem \ref{t2} are a
consequence of Proposition \ref{p6}. For the proof of the lower
bound, we obtain that
\[
\begin{split}
b_1^3-3b_1\,b_2+3b_3&-\frac{2}{(\alpha+3)(\alpha+5)}\,
\Big(n+\frac{2\alpha}{3}\Big)\Big(n-\frac{\alpha+1}{6}\Big)(b_1^2-2b_2)\\
&=\frac{1}{(\alpha+1)^3(\alpha+2)(\alpha+3)(\alpha+4)(\alpha+5)}
\;\sum_{j=1}^{5}\kappa_j(\alpha)\,n^{j}\,,
\end{split}
\]
with
\begin{eqnarray*}
&&\kappa_1(\alpha)=\frac{1}{270}\,(1+\alpha)^2(10\,\alpha^{3}
+100\,\alpha^2+321\,\alpha+1620)\,,\\
&&\kappa_2(\alpha)=\frac{1}{36}\,(1+\alpha)(4\,\alpha^4+35\,\alpha^{3}
+166\,\alpha^{2}+417\,\alpha+660)\,,\\
&&\kappa_3(\alpha)=\frac{1}{54}\,(4\,\alpha^5+36\,\alpha^{4}
+192\,\alpha^{3}+625\,\alpha^2+1527\,\alpha+1332)\,,\\
&&\kappa_4(\alpha)=\frac{1}{36}\,(\alpha^4-\alpha^{3}
+157\,\alpha^{2}+579\,\alpha+780)\,,\\
&&\kappa_5(\alpha)=\frac{1}{30}\,(\alpha^3+7\,\alpha^{2}
+136\,\alpha+280)\,.
\end{eqnarray*}
Obviously, $\,\kappa_j(\alpha)>0\,$  for $\,\alpha>-1$, $1\leq j\leq
5$, and hence the lower bound holds:
$$
\big[c_n(\alpha)\big]^2>\frac{b_1^3-3b_1\,b_2+3b_3}{b_1^2-2\,b_2}
>\frac{2}{(\alpha+3)(\alpha+5)}\,
\Big(n+\frac{2\alpha}{3}\Big)\Big(n-\frac{\alpha+1}{6}\Big)\,.
$$

For the proof of the upper bound for $c_n(\alpha)$ in Theorem
\ref{t2}, we find that
\[
\begin{split}
\frac{1}{(\alpha+1)^3(\alpha+3)(\alpha+5)}&\,
(n+1)^3\Big(n+\frac{2(\alpha+1)}{5}\Big)^{3}-\big(b_1^3-3b_1\,b_2+3b_3\big)\\
&=\frac{1}{(\alpha+1)^2(\alpha+2)(\alpha+3)(\alpha+4)(\alpha+5)}
\;\sum_{j=0}^{5}\nu_j(\alpha)\,n^{j}\,,
\end{split}
\]
where
\begin{eqnarray*}
&&\nu_0(\alpha)=\frac{8}{125}\,(1+\alpha)^2(2+\alpha)(4+\alpha)\,;\\
&&\nu_1(\alpha)=\frac{3}{250}\,(1+\alpha)(16\,\alpha^{3}
+152\,\alpha^2+439\,\alpha-52)\,,\\
&&\nu_2(\alpha)=\frac{1}{500}\,(96\,\alpha^4+1363\,\alpha^{3}
+5656\,\alpha^{2}+9167\,\alpha+2828)\,,\\
&&\nu_3(\alpha)=\frac{1}{250}\,(16\,\alpha^4+
363\,\alpha^{3}+2506\,\alpha^2+7167\,\alpha+4708)\,,\\
&&\nu_4(\alpha)=\frac{1}{100}\,(23\,\alpha^{3}
+446\,\alpha^{2}+1657\,\alpha+2164)\,,\\
&&\nu_5(\alpha)=\frac{3}{5}\,(5\alpha+16)\,.
\end{eqnarray*}

We shall show now that
\begin{equation}\label{e12}
\sum_{j=0}^{5}\nu_{j}(\alpha)\,n^{j}\geq 0\,,\qquad n\geq 2\,,\ \
\alpha>-1\,.
\end{equation}

Notice that, unlike the case with the coefficients
$\,\{\kappa_{j}(\alpha)\}_{j=1}^{5}$, which are all positive for all
admissible values of $\,\alpha\,$, i.e., $\alpha>-1\,$, here the
coefficients $\,\nu_j(\alpha)\,$, $1\leq j\leq 3$, assume negative
values for some $\alpha\in (-1,0)\,$ ($\nu_1(\alpha)$ is negative
also for some $\,\alpha>0$).

Since $\,\nu_4(\alpha)\,$ and $\,\nu_5(\alpha)\,$ are positive for
$\,\alpha>-1\,$,  for $\,n\geq 2$ we have
$$
\sum_{j=3}^{5}\nu_{j}(\alpha)\,n^{j}\geq\big(4\,\nu_5(\alpha)+2\,\nu_4(\alpha)
+\nu_3(\alpha)\big)n^3=:\widetilde{\nu}_3(\alpha)\,n^3\,,
$$
where
$$
\widetilde{\nu}_3(\alpha)=\frac{1}{125}\,(8\,\alpha^4+239\,\alpha^{3}
+2368\,\alpha^{2}+9226\,\alpha+12564)\,.
$$

Since $\,\widetilde{\nu}_3(\alpha)>0\,$ for $\,\alpha>-1$, we have
$$
\sum_{j=2}^{5}\nu_{j}(\alpha)\,n^{j}\geq
\big(2\widetilde{\nu}_3(\alpha)+\nu_2(\alpha)\big)\,n^2
=:\widetilde{\nu}_2(\alpha)\,n^2\,,\qquad n\geq 2\,,
$$
where
$$
\widetilde{\nu}_2(\alpha)=\frac{1}{100}\,(32\,\alpha^4+655\,\alpha^{3}
+4920\,\alpha^{2}+16595\,\alpha+20668)\,.
$$

Now, from $\,\widetilde{\nu}_2(\alpha)>0\,$ for $\,\alpha>-1$, we
obtain
$$
\sum_{j=1}^{5}\nu_{j}(\alpha)\,n^{j}\geq
\big(2\widetilde{\nu}_2(\alpha)+\nu_1(\alpha)\big)\,n
=:\widetilde{\nu}_1(\alpha)\,n\,,\qquad n\geq 2\,,
$$
with
$$
\widetilde{\nu}_1(\alpha)=\frac{1}{250}\,(160\,\alpha^4+3323\,\alpha^3
+25056\,\alpha^2+84292\,\alpha+103184)>0\,,\quad \alpha>-1\,.
$$
Hence, $\sum_{j=0}^{5}\nu_{j}(\alpha)\,n^{j}\geq
\widetilde{\nu}_1(\alpha)\,n+\nu_{0}(\alpha)>0\,$, and \eqref{e12}
is proved. From \eqref{e12} we conclude that
$$
\frac{1}{(\alpha+1)^3(\alpha+3)(\alpha+5)}\,
(n+1)^3\Big(n+\frac{2(\alpha+1)}{5}\Big)^{3}>b_1^3-3b_1\,b_2+3b_3\,,
$$
In view of Proposition \ref{p6}, the latter inequality proves the
upper bound for $c_n(\alpha)$ in Theorem \ref{t2}.

\section{Concluding Remarks} \label{s5}
\textbf{1.} Our main concern here is the major terms in the bounds
for the best Markov constant $c_n(\alpha)$, obtained through
Proposition \ref{p1}. We did not care much about the lower degree
terms, where perhaps some improvement is possible.\smallskip

\textbf{2.} Obviously, D\"{o}rfler's upper bound for $c_{n}(\alpha)$
in \eqref{e4} is a consequence of Proposition \ref{p1} (i).
D\"{o}rfler's lower bound for $c_{n}(\alpha)$ in \cite{PD1991},
which is slightly better than the one given in \eqref{e4}, is
obtained from Proposition \ref{p1} (ii). Both our lower and upper
bounds for the asymptotic constant $\,c(\alpha)$, given in
Corollary~\ref{c1}, are superior for all $\alpha>-1$ to
D\"{o}rfler's bounds obtained from \eqref{e4} .
\smallskip

\textbf{3.} The upper bounds for the largest zero $x_n$ of a
polynomial having only real and positive zeros in Proposition
\ref{p1} (ii) and (ii) admit some improvement. For instance, in
Proposition \ref{p1} (ii)  one can apply the quadratic mean --
arithmetic mean inequality to obtain
$$
b_1^2-2b_2=x_n^2+\sum_{i=1}^{n-1}x_i^2\ge
x_n^2+\frac{\Big(\sum_{i=1}^{n-1}|x_i|\Big)^2}{n-1}\ge
x_n^2+\frac{(b_1-x_n)^2}{n-1}\,,
$$
which yields a (slightly stronger) quadratic inequality for $x_n$
(actually, for any of the zeros of the polynomial $P$),
$$
n\,x_n^2-2b_1\,x_n+2(n-1)b_2-(n-2)b_1^2\le 0\,.
$$
The solution of the latter inequality,
$$
\frac{1}{n}\Big[b_1-\sqrt{(n-1)^2b_1^2-2(n-1)n\,b_2}\,\Big]\le x_n
\le \frac{1}{n}\Big[b_1+\sqrt{(n-1)^2b_1^2-2(n-1)n\,b_2}\,\Big]\,,
$$
provides lower and upper bounds for the zeros of an arbitrary
real-root monic polynomial of degree $n$ in terms of its two leading
coefficients $b_1$ and $b_2$. This result, due to Laguerre, is known
also as Laguerre-Samuelson inequality (for more details, see e.g.
\cite{STJ1999} and the references therein).
\smallskip

In a similar way one can obtain a slight improvement for the upper
bound in Proposition \ref{p1} (iii). However, in our case this
improvement is negligible (it affects only the lower degree terms in
the upper bound for $c_n(\alpha)$).

\begin{acknowledgement}
The research on this paper was conducted during a visit of the
first-named author to the Department of Applied Mathematics and
Theoretical Physics of the University of Cambridge in January, 2015.
The work was accomplished during a three week stay of the authors in
the Oberwolfach Mathematical Institute in April, 2016 within the
Research in Pairs Program. The first-named author acknowledges the
partial support by the Sofia University Research Fund through
Gontract no. 30/2016.
\end{acknowledgement}
%

%
%

\begin{thebibliography}{99.}%
\bibitem{SNA2015}
Aleksov, D., Nikolov, G., Shadrin, A.: On the Markov inequality in
the $L_2$ norm with the Gegenbauer weight. J. Approx. Theory (2016),
\url{http://dx.doi.org//10.1016/j.jat.2016.03.005}. See also:
arXiv:1510.03265v1 [math.CA]

\bibitem{BB1}
Bojanov, B.: Markov-type inequalities for polynomials and splines.
In: Chui, C.K., Schumaker, L.L., Stoeckler, J. (eds.) Approximation
Theory X. Abstract and Classical Analysis, pp. 31--90. Vanderbilt
University Press, Vanderbilt (2002)

\bibitem{BN1996}
Bojanov, B., Nikolov, G.: Duffin and Schaeffer type inequality for
ultrasphrical polynomials.  J. Approx. Theory  \textbf{84}, 129--138
(1996)


\bibitem{BD2010}
B\"{o}ttcher, A., D\"{o}rfler, P.: Weighted Markov-type
inequalities, norms of Volterra operators, and zeros of Bessel
functions. Math. Nachr. \textbf{283}, 357--367 (2010)

\bibitem{BD2010a}
B\"{o}ttcher, A., D\"{o}rfler, P.: On the best constant in
Markov-type inequalities involving Gegenbauer norms with different
weights. Oper. Matr. \textbf{161}, 40--57 (2010)

\bibitem{BD2011}
B\"{o}ttcher, A., D\"{o}rfler, P.: On the best constant in
Markov-type inequalities involving Laguerre norms with different
weights. Monatsh. Math.  \textbf{5}, 261--272 (2011)

\bibitem{PD1987}
D\"{o}rfler, P.: New inequalities of Markov type. SIAM J. Math.
Anal. \textbf{18}, 490--494 (1987)

\bibitem{PD1991}
D\"{o}rfler, P.: \"{U}ber die bestm\"{o}gliche Konstante in
Markov-Ungleichungen mit Laguerre Gewicht. \"{O}sterreich. Akad.
Wiss. Math.-Natur. Kl. Sitzungsber. II \textbf{200}, 13--20 (1991)

\bibitem{PD2002}
D\"{o}rfler, P.: Asymptotics of the best constant in a certain
Markov-type inequality. J. Approx. Theory \textbf{114}, 84--97
(2002)

\bibitem{AE2001}
Elbert, A.: Some recent results on the zeros of Bessel functions and
orthogonal polynomials. J. Comp. Appl. Anal. \textbf{133}, 65--83
(2001)

\bibitem{IS1985}
Ifantis, E.K., Siafarikas, P.D.: A differential equation for the
zeros of Bessel functions. Appl. Anal. \textbf{20}, 269--281 (1985)

\bibitem{STJ1999}
Jensen, S.T.: The Laguerre-Samuelson inequality with extensions and
applications in statistics and matrix theory. Master Thesis, McGill
University, Montreal, Canada, 1999. Available at:
\url{http://www-stat.wharton.upenn.edu/~stjensen/papers/shanejensen.mscthesis99.pdf}

\bibitem{AK2008}
Kro\'{o}, A.: On the exact constant in the $L_2$ Markov inequality.
J. Approx. Theory \textbf{151}, 208--211 (2008)

\bibitem{AM1889}
Markov, A.A.: On a question of D. I. Mendeleev. Zapiski Petersb.
Akad. Nauk \textbf{62}, 1--24 (1889) (in Russian). Available also
at: \url{http://www.math.technion.ac.il/hat/fpapers/mar1.pdf}

\bibitem{VM1892}
Markov, V.A.: On functions which deviate least from zero in a given
interval. Saint-Petersburg University, 1892 (in Russian); German
translation: Math. Ann. \textbf{77}, 213--258 (1916). Available also
at: \url{http://www.math.technion.ac.il/hat/fpapers/vmar.pdf}


\bibitem{MMR1994}
Milovanovi\'{c}, G.V., Mitrinovi\'{c}, D.S., Rassias, Th. M.: Topics
in Polynomials: Extremal Problems, Inequalities, Zeros. World
Scientific, Singapore (1994)

\bibitem{GN1998}
Nikolov, G.: On certain Duffin and Schaeffer type inequalities. J.
Approx. Theory {\bf 93}, 157--176 (1998)

\bibitem{GN2001}
Nikolov, G.: Inequalities of Duffin-Schaeffer type.  SIAM J. Math.
Anal. {\bf 33}, 686--698 (2001)

\bibitem{GN2003}
Nikolov, G.: Markov-type inequalities in the $L_2$-norms induced by
the Tchebycheff weights. Arch. Ineq. Appl. \textbf{1}, 361--376
(2003)

\bibitem{GN2005}
Nikolov, G.: Inequalities of Duffin-Schaeffer type. II. East J.
Approx. {\bf 11}, 147--168 (2005)

\bibitem{GN2005b}
Nikolov, G.: An extension of an inequality of Duffin and Schaeffer.
Constr. Approx. \textbf{21}, 181--191 (2004)

\bibitem{GN2005A}
Nikolov, G.: Polynomial inequalities of Markov and Duffin-Schaeffer
type. In: Bojanov, B. (ed.)  Constructive theory of functions, pp.
201--246. Prof. Marin Drinov Academic Publishing House, Sofia (2005)

\bibitem{NS2012}
Nikolov, G., Shadrin, A.: On Markov--Duffin--Schaeffer inequalities
with a majorant. In: Nikolov, G., Uluchev, R. (eds.) Constructive
theory of functions, Sozopol 2010, pp. 227--264. Prof. Marin Drinov
Academic Publishing House, Sofia (2012)

\bibitem{NS2014}
Nikolov, G., Shadrin, A.: On Markov--Duffin--Schaeffer inequalities
with a majorant. II. In: Ivanov, K, Nikolov, G., Uluchev, R. (eds.)
Constructive theory of functions, Sozopol 2013,  pp. 175--197. Prof.
Marin Drinov Academic Publishing House, Sofia (2014)

\bibitem{ES}
Schmidt, E.: \"{U}ber die nebst ihren Ableitungen orthogonalen
Polynomensysteme und das zugeh\"{o}rige Extremum. Math. Anal.
\textbf{119}, 165--204 (1944)

\bibitem{AS1992}
Shadrin, A.Yu.:  Interpolation with Lagrange polynomials. A simple
proof of Markov inequality and some of its generalizations Approx.
Theory Appl. \textbf{8}, 51--61 (1992)

\bibitem{AS2004}
Shadrin, A.: Twelve proofs of the Markov inequality. In: Dimitrov,
D.K., Nikolov, G., Uluchev, R. (eds.) Approximation Theory: A volume
dedicated to Borislav Bojanov, pp. 233--298. Professor Marin Drinov
Academic Publishing House, Sofia (2004). Available also at: \url{
http://www.damtp.cam.ac.uk/user/na/people/Alexei/papers/markov.pdf}

\bibitem{GS}
Szeg\H{o}, G.: Orthogonal polynomials. AMS Colloq. Publ.
\textbf{23}, AMS, Providence, RI (1975)

\bibitem{PT1960}
Tur\'{a}n, P.: Remark on a theorem of Ehrhard Schmidt, Mathematica
(Cluj) \textbf{2}, 373--378 (1960)

\bibitem{EvD1987}
van Doorn, E.A.: Representations and bounds for the zeros of
orthogonal polynomials and eigenvalues of sign-symmetric
tri-diagonal matrices. J. Approx Theory \textbf{51}, 254--266 (1987)

\end{thebibliography}
%

\end{document}